\documentclass[a4paper,11pt]{article}
\usepackage{amsmath,amsthm,amssymb,enumitem,xcolor}

\usepackage[hang]{footmisc}
\usepackage[nosort,nocompress,noadjust]{cite}

\usepackage[bookmarks=false,hyperfootnotes=false,colorlinks,
    linkcolor={red!60!black},
    citecolor={blue!50!black},
    urlcolor={blue!80!black}]{hyperref}

\renewcommand{\eqref}[1]{\hyperref[#1]{(\ref{#1})}}

\newlist{enumlist}{enumerate}{2}
\setlist[enumlist,1]{labelindent=0cm,label=\arabic*.,ref=\arabic*,labelwidth=2.5ex,labelsep=0.5ex,leftmargin=3ex,align=left,topsep=0.5ex,itemsep=1ex,parsep=1ex}
\setlist[enumlist,2]{labelindent=0cm,label=\theenumlisti.\arabic*.,ref=\arabic*,labelwidth=5ex,labelsep=0.5ex,leftmargin=5.5ex,align=left,topsep=0.5ex,itemsep=1ex,parsep=1ex}

\newlist{itemlist}{itemize}{1}
\setlist[itemlist]{labelindent=0cm,label=$\bullet$,labelwidth=2.5ex,labelsep=0.5ex,leftmargin=3ex,align=left,topsep=0.5ex,itemsep=1ex,parsep=1ex}

{\theoremstyle{definition}\newtheorem{definition}{Definition}%[section]
\newtheorem*{definition*}{Definition}
\newtheorem{letterdef}{Definition}

\newtheorem{remark}[definition]{Remark}
\newtheorem{example}[definition]{Example}
\newtheorem{question}[definition]{Question}
\newtheorem*{example*}{Example}
\newtheorem*{examples*}{Examples}}

\newtheorem{proposition}[definition]{Proposition}
\newtheorem{lemma}[definition]{Lemma}

\newtheorem{letterthm}[letterdef]{Theorem}
\newtheorem{lettercor}[letterdef]{Corollary}

{\theoremstyle{definition}}

\newcommand{\C}{\mathbb{C}}

\newcommand{\eps}{\varepsilon}
\newcommand{\al}{\alpha}
\newcommand{\be}{\beta}
\newcommand{\ot}{\otimes}

\newcommand{\Z}{\mathbb{Z}}
\newcommand{\vphi}{\varphi}
\newcommand{\cO}{\mathcal{O}}
\newcommand{\id}{\mathord{\text{\rm id}}}
\newcommand{\om}{\omega}
\newcommand{\N}{\mathbb{N}}
\newcommand{\ovt}{\mathbin{\overline{\otimes}}}

\newcommand{\si}{\sigma}
\newcommand{\R}{\mathbb{R}}
\newcommand{\F}{\mathbb{F}}
\newcommand{\cH}{\mathcal{H}}
\newcommand{\cZ}{\mathcal{Z}}
\newcommand{\Ad}{\operatorname{Ad}}
\newcommand{\cG}{\mathcal{G}}
\newcommand{\cK}{\mathcal{K}}

\newcommand{\actson}{\curvearrowright}

\newcommand{\cW}{\mathcal{W}}
\newcommand{\cU}{\mathcal{U}}

\newcommand{\cN}{\mathcal{N}}

\newcommand{\cV}{\mathcal{V}}
\newcommand{\cE}{\mathcal{E}}

\newcommand{\Aut}{\operatorname{Aut}}

\newcommand{\dpr}{^{\prime\prime}}
\newcommand{\Out}{\operatorname{Out}}
\newcommand{\Inn}{\operatorname{Inn}}
\newcommand{\bim}[3]{\mathord{\raisebox{-0.4ex}[0ex][0ex]{\scriptsize $#1$}{#2}\hspace{-0.25ex}\raisebox{-0.4ex}[0ex][0ex]{\scriptsize $#3$}}}
\newcommand{\otmax}{\otimes_{\text{\rm max}}}
\newcommand{\otalg}{\otimes_{\text{\rm alg}}}
\newcommand{\op}{^\text{\rm op}}

\newcommand{\SL}{\operatorname{SL}}

\newcommand{\Gammah}{\widehat{\Gamma}}

\begin{document}

\begin{center}
{\boldmath\LARGE\bf Spectral gap and strict outerness\vspace{0.5ex}\\ for actions of locally compact groups on full factors}

\bigskip

{\sc by Amine Marrakchi\footnote{UMPA, CNRS ENS de Lyon, Lyon (France). E-mail: amine.marrakchi@ens-lyon.fr} and Stefaan Vaes\footnote{\noindent KU~Leuven, Department of Mathematics, Leuven (Belgium). E-mail: stefaan.vaes@kuleuven.be.\\ S.V. is supported by FWO research project G090420N of the Research Foundation Flanders, and by long term structural funding~-- Methusalem grant of the Flemish Government.}}

\bigskip

In honor of Vaughan F.\ R.\ Jones.
\end{center}

\begin{abstract}\noindent
We prove that an outer action of a locally compact group $G$ on a full factor $M$ is automatically strictly outer, meaning that the relative commutant of $M$ in the crossed product is trivial. If moreover the image of $G$ in the outer automorphism group $\Out M$ is closed, we prove that the crossed product remains full. We obtain this result by proving that the inclusion of $M$ in the crossed product automatically has a spectral gap property. Such results had only been proven for actions of discrete groups and for actions of compact groups, by using quite different methods in both cases. Even for the canonical free Bogoljubov actions on free group factors or free Araki-Woods factors, these results are new.
\end{abstract}

\section{Introduction and main results}

To every strongly continuous action $G \actson^\al M$ of a locally compact group $G$ on a von Neumann algebra $M$ is associated the \emph{crossed product} von Neumann algebra $M \rtimes G$. Questions on how properties of the group action $G \actson^\al M$ are reflected by properties of the crossed product $M \rtimes G$ or the inclusion $M \subset M \rtimes G$ are among the most fundamental questions in operator algebras. When $G$ is a discrete group, several of the basic questions are easy to answer. For instance, for discrete groups $G$, we have that the relative $M' \cap M \rtimes G$ is as small as possible, i.e.\ equal to the center $\cZ(M)$, if and only if the action $G \actson^\al M$ is \emph{properly outer}: if $g \in G \setminus \{e\}$ and $b \in M$ are such that $b \al_g(a) = a b$ for all $a \in M$, then $b=0$. In particular, for an action $G \actson^\al M$ of a discrete group $G$ on a factor $M$, we have that $M' \cap M \rtimes G = \C 1$ if and only if the action $\al$ is \emph{outer}, meaning that $\al_g$ is an outer automorphism for every $g \in G \setminus \{e\}$.

When the acting group $G$ is no longer discrete, even these simple questions turn out to be quite subtle. Following \cite{Vae02}, a strongly continuous action $G \actson^\al M$ on a factor $M$ is called \emph{strictly outer} if $M' \cap M \rtimes G = \C 1$. It is easy to see that a strictly outer action $\al$ must be outer. The converse is however by no means true (see e.g.\ Example \ref{exam.counterexample} below). The main reason why crossed products by discrete groups are easier to study is the availability of a Fourier series decomposition for elements $x \in M \rtimes G$~: denoting by $E : M \rtimes G \to M$ the canonical conditional expectation and denoting by $(u_g)_{g \in G}$ the canonical unitary elements in $M \rtimes G$, we can define $(x)_g = E(x u_g^*)$ and formally write $x = \sum_{g \in G} (x)_g u_g$. When $G$ is no longer discrete, there typically is no normal conditional expectation of $M \rtimes G$ onto $M$ and there is no way to define a Fourier series decomposition.

To illustrate the subtleness of strict outerness, we mention here one of the fundamental results in modular theory for von Neumann algebras: by \cite[Theorem 5.1]{CT76}, a trace scaling action $\R \actson^\theta N$ of the group $\R$ on a II$_\infty$ factor $N$ is strictly outer. Because an inner automorphism is obviously trace preserving, a trace scaling action is outer, but its strict outerness is one of the deepest results in modular theory, established in \cite{CT76}.

In our first main result, we prove that outerness and strict outerness are equivalent for actions of a locally compact second countable (lcsc) group $G$ on a \emph{full} factor with separable predual. Recall that a factor $M$ with separable predual is called full if the group $\Inn M$ of inner automorphisms of $M$ is closed in the group of all automorphisms $\Aut M$. When $G$ is discrete, this result is trivial, as mentioned above. When $G$ is compact, this result was obtained recently in \cite{BMO19} using a method that is very specific for compact groups, making use of the subalgebra $M^G$ of $G$-invariant elements, which could very well be trivial when $G$ is noncompact.

Our second main result concerns the fullness of the crossed product $M \rtimes G$ when $G$ is a lcsc group acting on a full factor $M$ with separable predual. For such full factors, one considers the Polish group $\Out M = \Aut M / \Inn M$. When $G$ is discrete and $M$ is a II$_1$ factor, this question was solved by Vaughan Jones in \cite[Theorem 1]{Jon81}. Jones proved that if $G \actson^\al M$ is an outer action of a discrete group $G$ on a II$_1$ factor $M$ and if the image of $G$ in $\Out M$ is closed, then $M \rtimes G$ is full. More precisely, Jones proved that the inclusion $M \subset M \rtimes G$ has $w$-spectral gap (see below) if and only if the image of $G$ in $\Out M$ is closed. For arbitrary full factors $M$, not necessarily of type II$_1$, the same equivalence was proven when $G$ is discrete in \cite{Mar16,HMV16} (see also \cite{Mar18a}), and when $G$ is compact in \cite{BMO19}. The methods that were used in the discrete case and in the compact case were rather disjoint. In our second main result, we prove the equivalence in complete generality, for arbitrary full factors and locally compact groups, concluding the line of research that started in \cite{Jon81}.

Let $N \subset M$ be an inclusion of von Neumann algebras with separable predual. Generalizing \cite[Section 2]{Pop09} beyond the II$_1$ setting, we say that $N \subset M$ has \emph{$w$-spectral gap} if for every bounded sequence $a_n \in M$ satisfying $a_n b - b a_n \to 0$ $*$-strongly for all $b \in N$, there exists a bounded sequence $d_n \in N' \cap M$ satisfying $a_n - d_n \to 0$ $*$-strongly. We say that $N \subset M$ has \emph{stable $w$-spectral gap} if $1 \ot N \subset B(\ell^2(\N)) \ovt M$ has $w$-spectral gap.

\begin{letterthm}\label{thm.main}
Let $G \actson^\al M$ be a strongly continuous action of a lcsc group $G$ on a factor $M$ with separable predual. Assume that $M$ is full. Denote by $\pi : \Aut M \to \Out M$ the quotient homomorphism.
\begin{enumlist}
\item If the action $\al$ is outer, then $\al$ is strictly outer.
\item The following three properties are equivalent.
\begin{enumlist}[label=(\alph*)]
\item\label{stat.a} The action $\al$ is outer and the subgroup $\pi(\al(G))$ of $\Out M$ is closed.
\item\label{stat.b} The action $\al$ is strictly outer and $M \subset M \rtimes G$ has $w$-spectral gap.
\item\label{stat.c} The action $\al$ is strictly outer and $M \subset M \rtimes G$ has stable $w$-spectral gap.
\end{enumlist}
In particular, if $\al$ is outer and $\pi(\al(G))$ is a closed subgroup of $\Out M$, then $M \rtimes G$ is a full factor.
\end{enumlist}
\end{letterthm}

Given an ultrafilter $\om$ on $\N$, we denote by $P^\om$ the Ocneanu ultrapower of a von Neumann algebra $P$ with separable predual. So, whenever $G \actson^\al M$ is an outer action on a full factor with $\pi(\al(G))$ being a closed subgroup of $\Out M$, Theorem \ref{thm.main} says in particular that $M' \cap (M \rtimes G)^\om = \C 1$. But the result of Theorem \ref{thm.main} is stronger, because it says that \emph{all} bounded sequences $a_n \in M \rtimes G$ that asymptotically commute with $M$ must be asymptotically scalar, without requiring that $(a_n)_{n \in \N}$ belongs to $(M \rtimes G)^\om$.

A wide class of examples to which Theorem \ref{thm.main} applies is given by the free Bogoljubov actions on free Araki-Woods factors. To every orthogonal representation $U : \R \actson H_\R$ is associated Shlyakhtenko's free Araki-Woods factor $M=\Gamma(H_\R,U)\dpr$, see \cite{Shl96}. Whenever $\dim_\R H_\R \geq 2$, the von Neumann algebra $M$ is a full factor. When $U$ is the trivial representation, this construction coincides with Voiculescu's free Gaussian functor and $M \cong L(\F_{\dim_\R H_\R})$. The construction is functorial: to every orthogonal transformation $v \in \cO(H_\R)$ that commutes with $U$ is associated the free Bogoljubov automorphism $\al_v$ of $M$.

We can then use Theorem \ref{thm.main} to provide the following sufficient condition for the fullness of a free Bogoljubov crossed product. For discrete groups, this result was proven in \cite[Theorem A]{Hou12} for the free Gaussian case, and in \cite[Theorem B]{HT18} for free Araki-Woods factors.

\begin{lettercor}\label{cor.free-araki-woods}
Let $U : \R \actson H_\R$ be a strongly continuous action of $\R$ by orthogonal transformations of a separable real Hilbert space $H_\R$ with $\dim_\R H_\R \geq 2$. Denote by $M=\Gamma(H_\R,U)\dpr$ the associated free Araki-Woods factor. Let $G$ be a lcsc group. To every strongly continuous orthogonal representation $\pi : G \to \cO(H_\R)$ that commutes with $U$ is associated the canonical action $G \actson^{\al_\pi} M$ by free Bogoljubov automorphisms.
\begin{enumlist}
\item If $\pi$ is faithful, then $\al_\pi$ is strictly outer.
\item If $\pi$ is faithful and $\pi(G) \subset \cO(H_\R)$ is closed, then $M \subset M \rtimes_{\al_\pi} G$ has stable $w$-spectral gap. In particular, $M \rtimes_{\al_\pi} G$ is a full factor.
\end{enumlist}
\end{lettercor}

Note that the hypotheses of point~2 in Corollary \ref{cor.free-araki-woods} are automatically satisfied when $\pi$ is a faithful and mixing representation, e.g.\ when $\pi$ is the left regular representation on $L^2_\R(G)$. More generally, if there exists a unit vector $\xi \in H_\R$ such that $\langle \pi(g) \xi,\xi\rangle \to 0$ whenever $g \to \infty$ in $G$, then $\pi(G)$ is a closed subgroup of $\cO(H_\R)$. So, when $G$ satisfies the Howe-Moore property, e.g.\ when $G$ is a connected noncompact simple real Lie group with finite center like $\SL(n,\R)$, then for every orthogonal representation $\pi$ of $G$, the subgroup $\pi(G)$ is closed.

It is very natural to try to generalize Theorem \ref{thm.main} to factors $M$ that are not full, replacing $\Out M$ in the hypotheses by the Polish group $\Aut M / \overline{\Inn M}$. We discuss this problem in Section \ref{final.sec} and relate this to several open questions and partial results for outer actions of locally compact groups on factors $M$ that are no longer full.

%{\bf ... Add a remark that this is automatic when $\pi$ has a mixing component. Hence it is automatic for all nontrivial representations of a group with the Howe-Moore property. Mention examples ...}

%{\bf ... Add a remark that this also provides a new proof of a result of Houdayer-Trom. Maybe also of Houdayer-... in the II$_1$ setting. ...}

\section{Proof of Theorem \ref{thm.main}}

To prove Theorem \ref{thm.main}, we can no longer use Fourier decompositions as in the discrete case, and we can no longer use the co-amenability approach of \cite{BMO19} which is too specific to the compact case. Instead, we make substantial use of Hilbert bimodules. Fix a factor $M$ with separable predual.

Recall that a Hilbert $M$-$M$-bimodule $\bim{M}{\cH}{M}$ gives rise to the unital $*$-representation $\pi_\cH$ of $M \otmax M\op$ on $\cH$ given by
$$\pi_\cH(a \ot b\op)\xi = a \xi b \; .$$
Recall that $\bim{M}{\cH}{M}$ is said to be weakly contained in $\bim{M}{\cK}{M}$ if $\|\pi_\cH(T)\| \leq \|\pi_\cK(T)\|$ for all $T \in M \otmax M\op$.

Denote by $L^2(M)$ the standard Hilbert space of $M$, which we view as the trivial $M$-$M$-bimodule. View $\Aut M$ as a Polish group. Whenever $(X,\mu)$ is a standard $\si$-finite measure space and $\rho : X \to \Aut M$ is a Borel map, we define the $M$-$M$-bimodule $\cH(\rho)$ by
\begin{equation}\label{eq.Hrho}
\cH(\rho) = L^2(M) \ot L^2(X,\mu) \quad\text{with}\;\; (a \cdot \xi \cdot b)(x) = \rho(x)^{-1}(a) \, \xi(x) \, b \quad\text{for all $a,b \in M$, $x \in X$.}
\end{equation}
When $X$ is a singleton, and $\rho$ is thus determined by a single $\al \in \Aut M$, we denote by $\cH(\al)$ the corresponding $M$-$M$-bimodule. Its carrier Hilbert space is $L^2(M)$ and the bimodule action is given by $a \cdot \xi \cdot b = \al^{-1}(a) \xi b$.

The following is the main technical lemma that is needed to prove Theorem \ref{thm.main}.

\begin{lemma}\label{lem.amine}
Let $M$ be a full factor with separable predual. Denote by $\pi : \Aut M  \to \Out M$ the quotient homomorphism. Let $(X,\mu)$ be a standard $\si$-finite measure space and let $\rho : X \to \Aut M$ be a Borel map.

If $\al \in \Aut M$ is such that $\pi(\al)$ does not belong to the support of $(\pi \circ \rho)_*\mu$ inside $\Out M$, meaning that there exists an open set $\cV \subset \Out M$ with $\pi(\al) \in \cV$ and $\mu((\pi \circ \rho)^{-1}(\cV)) = 0$, then $\cH(\al)$ is not weakly contained in $\cH(\rho)$.
\end{lemma}

\begin{proof}
Since $\cH(\al^{-1}) \ot_M \cH(\al) \cong L^2(M)$ and $\cH(\al^{-1}) \ot_M \cH(\rho) \cong \cH(\rho')$ where $\rho'(x) = \al^{-1} \circ \rho(x)$, we may assume that $\al = \id$. We can then choose a neighborhood $\cV$ of the identity in $\Out M$ such that $\pi(\rho(x)) \in \Out M \setminus \cV$ for $\mu$-almost every $x \in X$. We choose $\cV$ to be symmetric, i.e.\ $\cV = \cV^{-1}$.

We distinguish three cases according to the type of $M$.

\textbf{\boldmath Type II$_1$}. By \cite[Lemma 4]{Jon81}, we can take $\varepsilon > 0$ and a finite subset $F \subset M$ such that for all $\theta \in \Aut M \setminus \pi^{-1}(\cV)$ and all $b \in M$, we have
$$ \varepsilon \| b \|_2^{2} \leq \sum_{a \in F} \| ba - \theta(a) b \|_2^{2}.$$
By taking $\theta=\rho(x)^{-1}$, $x \in X$, and integrating this inequality with respect to $\mu$, we obtain
$$ \varepsilon \| \xi \|^2  \leq \sum_{a \in F} \| \xi \cdot a - a\cdot \xi \|^{2}$$
for all $\xi \in \cH(\rho)$. But, since $L^2(M)$ has an $M$-central vector, if $L^2(M)$ was weakly contained in $\cH(\rho)$, we would find a sequence of unit vectors $\xi_n \in \cH(\rho)$ such that $\| \xi_n \cdot a - a \cdot \xi_n \| \to 0$. This contradicts the inequality above. We conclude that $L^2(M)$ is not weakly contained in $\cH(\rho)$.

\textbf{\boldmath Type II$_\infty$}. Take $p \in M$ such that $\tau(p)=1$ where $\tau$ is the semifinite trace of $M$. By \cite[Theorem 3.3]{HMV16}, we can take $\varepsilon > 0$ and a finite subset $F \subset pMp$ such that for all $\theta \in \Aut M \setminus \pi^{-1}(\cV)$ and all $b \in \theta(p)Mp$, we have
$$ \varepsilon \| b \|_2^{2} \leq \sum_{a \in F} \| ba - \theta(a) b \|_2^{2}.$$
By taking $\theta=\rho(x)^{-1}$, $x \in X$, and integrating this inequality with respect to $\mu$, we obtain
$$ \varepsilon \| \xi \|^2  \leq \sum_{a \in F} \| \xi \cdot a - a\cdot \xi \|^{2}$$
for all $\xi \in p\cdot \cH(\rho) \cdot p$. If $L^2(M)$ was weakly contained in $\cH(\rho)$, then $pL^2(M)p=L^2(pMp)$ would also be weakly contained in $p \cdot \cH(\rho) \cdot p$ as $pMp$-$pMp$-bimodules. Thus, we would find a sequence of unit vectors $\xi_n \in p \cdot \cH(\rho) \cdot p$ such that $\| \xi_n \cdot a - a \cdot \xi_n \| \to 0$. This contradicts the previous inequality. We conclude that $L^2(M)$ is not weakly contained in $\cH(\rho)$.

\textbf{Type III}.
This case is much more delicate and we need to use the technique of \cite{Mar18b}. For every faithful normal state $\om$ on $M$, we denote by $\xi_\om \in L^2(M)$ the canonical positive vector that implements $\om$. Given two faithful normal states $\psi$ and $\vphi$ on $M$, we denote by $\Delta_{\psi,\vphi}$ the relative modular operator (see \cite[Theorem VIII.3.2 and Lemma IX.1.5]{Tak03}). Recall that $\Delta_{\psi,\vphi}$ is a positive nonsingular operator and that $\Delta_{\psi,\vphi}^{1/2}$ is the closure of the linear map $b \xi_\vphi \mapsto \xi_\psi b$, $b \in M$. Also, using Connes' cocycle derivative,
$$\Delta_{\psi,\vphi}^{it} = [D\psi : D\vphi]_t \, \Delta_\vphi^{it} = J [D\vphi : D\psi]_t J \, \Delta_\psi^{it} \; .$$
In particular, $\Delta_{\psi,\vphi}^{it} b \Delta_{\psi,\vphi}^{-it} = \si_t^\psi(b)$ and $\Delta_{\psi,\vphi}^{it} JbJ \Delta_{\psi,\vphi}^{-it} = J \si_t^\vphi(b) J$ for all $b \in M$, $t \in \R$. Given a state $\vphi$ on $M$ and $\theta \in \Aut M$, we write $\theta(\vphi) = \vphi \circ \theta^{-1}$.

By \cite[Theorem 3.2]{HMV16}, there exists a faithful normal state $\varphi \in M_*$, a constant $\varepsilon > 0$ and a finite set $F \subset M$ with $a \xi_\vphi = \xi_\vphi a^*$ for all $a \in F$ such that for all $b \in M$ and all $\theta \in \Aut M \setminus \pi^{-1}(\cV)$, we have
$$ \varepsilon \| b \xi_\vphi\|^{2} \leq \sum_{a \in F} \| ba \xi_\vphi - \theta(a) b \xi_\vphi \|^{2} + \inf_{ \lambda \in \R_+} \| b \xi_\vphi -\lambda \xi_{\theta(\varphi)} b \|^{2} \; .$$
Denote $\Delta_\theta = \Delta_{\theta(\vphi),\vphi}$. It follows that
$$ \varepsilon \| \xi \|^2  \leq \sum_{a \in F} \| \xi a^* - \theta(a) \xi \|^{2} + \inf_{ \lambda \in \R_+} \| \xi -\lambda \Delta_\theta^{1/2} \xi \|^{2}$$
for all $\theta \in \Aut M \setminus \pi^{-1}(\cV)$ and all $\xi$ in the domain of $\Delta_\theta^{1/2}$.

Let $\Delta$ be the decomposable operator on $\cH(\rho)$ obtained by integrating $x \mapsto \Delta_{\rho(x)^{-1}}$. Then, we can integrate the previous inequality to obtain
\begin{equation}\label{eq.good-inequal}
\varepsilon \| \xi \|^2  \leq \sum_{a \in F} \| \xi \cdot a^* - a\cdot \xi \|^{2} + \inf_{ \lambda \in \R_+} \| \xi -\lambda \Delta^{1/2} \xi \|^{2}
\end{equation}
for all $\xi \in \cH(\rho)$ in the domain of $\Delta^{1/2}$.

Assume that $L^2(M)$ is weakly contained in $\cH(\rho)$. We then find a state $\om_\vphi$ on $B(\cH(\rho))$ such that
$$\om_\vphi(\pi_{\cH(\rho)}(a \ot b\op)) = \langle a \xi_\vphi b,\xi_\vphi\rangle \quad\text{for all}\quad a,b \in M \; .$$
Since $\Delta_\theta^{it} \theta(b) \Delta_\theta^{-it} = \si_t^{\theta(\vphi)}(\theta(b)) = \theta(\si_t^\vphi(b))$, while $\Delta_\theta^{it} J b J \Delta_\theta^{-it} = J \si_t^\vphi(b) J$ for all $b \in M$, $t \in \R$ and $\theta \in \Aut M$, we get that
$$\Delta^{it} \, \pi_{\cH(\rho)}(a \ot b\op) \, \Delta^{-it} = \pi_{\cH(\rho)}(\si_t^\vphi(a) \ot \si_t^\vphi(b)\op) \; .$$
We also have that $\langle \si_t^\vphi(a) \xi_\vphi \si_t^\vphi(b), \xi_\vphi\rangle = \langle  a \xi_\vphi b, \xi_\vphi\rangle$. We can thus apply the method of \cite[Lemma 4.1]{Mar18b} and may assume that $\om_\vphi$ is strongly invariant under $\Ad \Delta^{it}$.

By \cite[Theorem 3.2]{Mar18b}, this means that $\om_\vphi$ belongs to the weak$^*$ closed convex hull of $\cE \subset B(\cH(\rho))^*$, where $\cE$ is defined as the set of states $\om$ on $B(\cH(\rho))$ for which there exists a net of unit vectors $\xi_i \in \cH(\rho)$ and a net of positive numbers $\lambda_i > 0$ such that $\xi_i$ belongs to the domain of $\Delta^{1/2}$ for every $i$ and
\begin{equation}\label{eq.my-eq}
\langle T \xi_i, \xi_i \rangle \to \om(T) \;\;\text{for all $T \in B(\cH(\rho))$, and}\quad \|\xi_i - \lambda_i \Delta^{1/2} \xi_i\| \to 0 \; .
\end{equation}
Define the operator $T \in B(\cH(\rho))$ by
$$T=\sum_{a \in F} \pi_{\cH(\rho)}\bigl(a \ot 1 - 1 \ot (a^*)\op\bigr)^* \;\; \pi_{\cH(\rho)} \bigl(a \ot 1 - 1 \ot (a^*)\op\bigr) \; .$$
Since $a \xi_\vphi = \xi_\vphi a^*$ for all $a \in F$, we get that $\om_\vphi(T) = 0$. Since $\om_\vphi$ belongs to the weak$^*$ closed convex hull of $\cE$, we can take $\om \in \cE$ such that $\om(T) \leq \eps/2$, where $\eps > 0$ is as in \eqref{eq.good-inequal}. Take a net of unit vectors $\xi_i \in \cH(\rho)$ and a net of positive numbers $\lambda_i > 0$ such that $\xi_i$ belongs to the domain of $\Delta^{1/2}$ for every $i$ and such that \eqref{eq.my-eq} holds. It follows that
$$\lim_i \sum_{a \in F}  \| \xi_i \cdot a^* - a\cdot \xi_i \|^{2} = \om(T) \leq \eps/2 \quad\text{and}\quad \|\xi_i - \lambda_i \Delta^{1/2} \xi_i\| \to 0 \; .$$
But then \eqref{eq.good-inequal} leads to the contradiction that $\eps \leq \eps / 2$.
%
%Moreover, by the idea of \cite[Lemma 4.1]{Mar18b}, we can assume that $\omega_\varphi$ is strongly $\sigma$-invariant where $\sigma$ is the flow on $B(\cH(\rho))$ given by $\sigma_t=\Ad(\Delta^{\rm it})$. By \cite[Theorem 3.2]{Mar18b}, this means that $\omega_\varphi \in \overline{\conv} \mathcal{E}(\log \Delta)$ (See \cite[Section 3]{Mar18b} for all the relevant definitions and notations). By construction, we have $\omega_\varphi(\pi_{\cH(\rho)}(T))=0$ where
%$$T=\sum_{a \in F} | a \otimes 1-1 \otimes (a^*)\op |^2\; .$$
%Since $\omega_\varphi \in \conv \mathcal{E}(\log \Delta)$, we can find a state $\psi \in  \mathcal{E}(\log \Delta)$ such that $\psi(\pi_{\cH(\rho)}(T))=0$. By definition of  $\mathcal{E}(\log \Delta)$, we can then find a net of unit vectors $\xi_i$ in the domain of $\Delta$ and a net $\lambda_i \in \R^*_+$ such that $\lim_i\| \xi_i- \lambda_i \Delta \xi_i\|=0$ and
%$$  \lim_i \sum_{a \in F} \| \xi_i \cdot a^* - a\cdot \xi_i \|^{2} = \lim_i \langle \pi_{\cH(\rho)}(T) \xi_i,\xi_i \rangle = \psi(\pi_{\cH(\rho)}(T))=0 \; .$$
%This contradicts the inequality above.
\end{proof}

The following lemma is an easy observation.

\begin{lemma}\label{lem.easy}
Let $M$ be a factor with separable predual. Let $X$ be a locally compact second countable Hausdorff space and let $\rho : X \to \Aut M$ be a continuous map. Let $\mu$ be a $\sigma$-finite Borel measure on $X$. If $x \in X$ belongs to the support of $\mu$, meaning that $\mu(\cU) > 0$ for all open sets $\cU$ with $x \in \cU$, then $\cH(\rho(x))$ is weakly contained in $\cH(\rho)$.
\end{lemma}
\begin{proof}
Choose a decreasing family of open sets $\cU_n \subset X$ that form a neighborhood basis of $x \in X$. Since $\mu(\cU_n) > 0$, we can choose Borel sets $\cV_n \subset \cU_n$ such that $0 < \mu(\cV_n) < +\infty$. For every $n$, consider the unit vector $\eta_n \in L^2(X,\mu)$ given by $\eta_n = \mu(\cV_n)^{-1/2} 1_{\cV_n}$. Whenever $\xi \in L^2(M)$ and $a,b \in M$, it follows from the continuity of $\rho$ that
$$\lim_n \langle \pi_{\cH(\rho)}(a \ot b\op) (\xi \ot \eta_n) , \xi \ot \eta_n \rangle = \langle \pi_{\cH(\rho(x))}(a \ot b\op) \xi , \xi \rangle \; .$$
Therefore, $\|\pi_{\cH(\rho(x))}(T) \xi\| \leq \|\pi_{\cH(\rho)}(T)\| \, \|\xi\|$, for every $T \in M \otmax M\op$ and $\xi \in L^2(M)$, so that $\cH(\rho(x))$ is weakly contained in $\cH(\rho)$.
\end{proof}

Also the following lemma is straightforward.

\begin{lemma}\label{lem.compact}
Let $M$ be a von Neumann algebra, $H$ a Hilbert space and $a_n \in M \ovt B(H)$ a bounded sequence satisfying $a_n (1 \ot T) - (1 \ot T) a_n \to 0$ $*$-strongly for every compact operator $T \in \cK(H)$.

Then for any unit vector $\xi \in H$ with associated vector state $\om_\xi$ on $B(H)$, the bounded sequence $d_n = (\id \ot \om_\xi)(a_n)$ in $M$ satisfies $a_n - d_n \ot 1 \to 0$ $*$-strongly.
\end{lemma}
\begin{proof}
Fix a unit vector $\xi \in H$. Assume that $M \subset B(K)$. Take $\mu \in K$ and $\eta \in H$ arbitrary. Define the rank one operator $T \in \cK(H)$ by $T(\zeta) = \langle \zeta, \xi \rangle \, \eta$. Then,
$$a_n (\mu \ot \eta) = a_n(1 \ot T) \, (\mu \ot \xi) \quad\text{and}\quad (d_n \ot 1) (\mu \ot \eta) = (1 \ot T) a_n \, (\mu \ot \xi) \; .$$
It follows that $\|(a_n - d_n \ot 1)(\mu \ot \eta)\| \to 0$. So, $a_n - d_n \ot 1 \to 0$ strongly. By symmetry, also $a_n^* - d_n^* \ot 1 \to 0$ strongly and the lemma is proven.
\end{proof}

We are then ready to prove Theorem \ref{thm.main}.

\begin{proof}[{Proof of Theorem \ref{thm.main}}]
We define $\al : M \to M \ovt L^\infty(G) : \al(a)(g) = \al_{g^{-1}}(a)$, so that $M \rtimes G$ is realized as the von Neumann subalgebra of $M \ovt B(L^2(G))$ generated by $\al(M)$ and the unitary operators $1 \ot \lambda_g$, $g \in G$, where $\lambda_g$ is the left translation by $g$. We thus denote by $\al(M)$ the canonical copy of $M$ inside $M \rtimes G$. We also denote by $\lambda$ the left Haar measure on $G$.

For every nonnegligible Borel set $A \subset G$, we denote by $\cH(A)$ the $M$-$M$-bimodule associated with the map $A \to \Aut M : g \mapsto \al_g$ and the Haar measure $\lambda$ on $A$, as in \eqref{eq.Hrho}. We denote by $\pi_A : M \otmax M\op \to B(\cH(A))$ the corresponding $*$-homomorphism.

1.\ Write $\cN = \al(M)' \cap M \ovt B(L^2(G))$. We claim that $\cN = 1 \ot L^\infty(G)$. Note that $1 \ot L^\infty(G) \subset \cN$ and
\begin{equation}\label{eq.inclusions}
(1 \ot L^\infty(G))' \cap \cN = \al(M)' \cap M \ovt L^\infty(G) = 1 \ot L^\infty(G) \; ,
\end{equation}
because $M$ is a factor. In particular, $\cZ(\cN) \subset 1 \ot L^\infty(G)$. To prove the claim, it thus suffices to show that $1 \ot L^\infty(G) \subset \cZ(\cN)$, because it then follows from \eqref{eq.inclusions} that
$$\cN \subset \cZ(\cN)' \cap \cN \subset (1 \ot L^\infty(G))' \cap \cN \subset 1 \ot L^\infty(G) \; .$$
Fix a compact subset $K \subset G$. It thus suffices to prove that $1 \ot 1_K \in \cZ(\cN)$.

Since $\cZ(\cN) \subset 1 \ot L^\infty(G)$, we can take a Borel set $A \subset G$ such that $K \subset A$ and such that $1 \ot 1_A$ is the central support of $1 \ot 1_K$ in $\cN$. We have to prove that $\lambda(A \setminus K) = 0$. Define the closed set $L \subset G$ as the support of the measure $\lambda|_A$ on the locally compact second countable space $G$, meaning that $L$ is the smallest closed subset of $G$ with the property that $\lambda(A \setminus L) = \lambda|_A(G \setminus L) = 0$. Choose $g \in L$. By Lemma \ref{lem.easy}, the $M$-$M$-bimodule $\cH(g)$ is weakly contained in $\cH(A) \cong \cH(G,\lambda|_A)$. Viewing $\cN$ as the algebra of bounded $M$-$M$-bimodular maps on $\cH(G)$, it follows that the $M$-$M$-bimodules $\cH(K)$ and $\cH(A)$ are quasi-equivalent. Thus, $\cH(g)$ is weakly contained in $\cH(K)$. Since $K \subset G$ is compact, $(\pi \circ \al)(K) \subset \Out M$ is compact, and thus closed. By Lemma \ref{lem.amine}, we must have that $\pi(\al_g) \in \pi(\al(K))$. Since $\pi \circ \al$ is injective, it follows that $g \in K$. Since $g \in L$ was arbitrary, we have proven that $L \subset K$. Since $\lambda(A \setminus L) = 0$, also $\lambda(A \setminus K) = 0$ and the claim is proven.

To prove the first part of the theorem, take $T \in \al(M)' \cap M \rtimes G$. By the claim above, $T = 1 \ot F$ for some $F \in L^\infty(G)$. Denote by $U_g \in \cU(L^2(M))$ the canonical unitary implementation of the automorphism $\al_g$. By definition, the elements of $M \rtimes G$ commute with the unitary operators $U_h \ot \rho_h$, where $\rho_h$ denotes the right regular representation on $L^2(G)$. It follows that $F \in L^\infty(G)$ commutes with all right translations, so that $F \in \C 1$.

2.\ The implication \ref{stat.c} $\Rightarrow$ \ref{stat.b} is trivial. To prove that \ref{stat.b} $\Rightarrow$ \ref{stat.a}, assume that \ref{stat.a} does not hold. If $\al$ is not outer, we find a unitary $a \in \cU(M)$ and an element $g \in G \setminus \{e\}$ such that $d = \al(a) (1 \ot \lambda_g)$ commutes with $\al(M)$. Then the constant sequence $d$ says that \ref{stat.b} does not hold. If $\al$ is outer, but $\pi(\al(G))$ is not closed in $\Out M$, we can choose a sequence $g_n \to \infty$ in $G$ such that $\pi(\al_{g_n}) \to \id$ in $\Out M$. So we find a sequence $a_n \in \cU(M)$ such that $\be_n = \Ad a_n \circ \al_{g_n} \to \id$ in $\Aut M$. Write $d_n = \al(a_n) (1 \ot \lambda_{g_n})$. Let $b \in M$ be arbitrary. Since $\be_n \to \id$ in $\Aut M$, we have that $\be_n^{-1}(b) - b \to 0$ strongly. Then also
$$\al(b) d_n - d_n \al(b) = d_n \, \al(\be_n^{-1}(b) - b) \to 0 \quad\text{strongly.}$$
Also, $\be_n(b^*) \to b^*$ strongly, so that
$$(\al(b) d_n - d_n \al(b))^* = d_n^* \, \al(b^* - \be_n(b^*)) \to 0 \quad\text{strongly.}$$
We have thus proven that $d_n \al(b) - \al(b) d_n \to 0$ $*$-strongly for all $b \in M$. Assume that $t_n \in \C$ such that $d_n - t_n 1 \to 0$ $*$-strongly. Choose an arbitrary unit vector $\xi \in L^2(M)$ and an arbitrary compact neighborhood $K$ of $e$ in $G$. Since $g_n \to \infty$, we get that $g_n K \cap K = \emptyset$ for all $n$ large enough. So, for large enough $n$, we get that $\langle d_n (\xi \ot 1_K), \xi \ot 1_K \rangle = 0$. We conclude that $t_n \to 0$. So, $d_n \to 0$ $*$-strongly. This is impossible because $d_n$ is a sequence of unitary operators. So, \ref{stat.b} does not hold.

We now prove the difficult implication \ref{stat.a} $\Rightarrow$ \ref{stat.c}. So, assume that $\al$ is outer and that $\pi(\al(G))$ is closed in $\Out M$. Let $H$ be a separable Hilbert space and $a_n \in B(H) \ovt (M \rtimes G)$ a sequence such that $\|a_n\| \leq 1$ for all $n$ and $a_n (1 \ot \al(b)) - (1 \ot \al(b)) a_n \to 0$ $*$-strongly for all $b \in M$. We have to prove the existence of a bounded sequence $f_n \in B(H)$ such that $a_n - f_n \ot 1 \to 0$ $*$-strongly. We again make use of the bimodules $\cH(A)$ as introduced in the beginning of the proof, together with the corresponding $*$-representations $\pi_A$ of $M \otmax M\op$ on $L^2(M) \ot L^2(A)$.

Since $G$ and $\Out M$ are Polish groups and $\pi \circ \al : G \to \Out M$ is a continuous homomorphism with closed range, the map $\pi \circ \al$ is closed.

{\bf Step A.} We prove that for every open subset $\cU \subset G$ with complement $L = G \setminus \cU$, the sequence $X_n \in B(H) \ovt M \ovt B(L^2(G))$ defined by $X_n = (1 \ot 1 \ot 1_L) a_n (1 \ot 1 \ot 1_\cU)$ converges to $0$ strongly.

Fix a unit vector $\eta \in H$ and write $Y_n = X_n (\eta \ot 1 \ot 1)$. Note that $\|Y_n\| \leq 1$ for all $n$. We have to prove that $Y_n^* Y_n$ converges to $0$ weakly. Assume that this is not the case. After passage to a subsequence, we may assume that $Y_n^* Y_n$ converges weakly to a nonzero $S \in M \ovt B(L^2(\cU))$. Because $a_n (1 \ot \al(a)) - (1 \ot \al(a)) a_n \to 0$ $*$-strongly for all $a \in M$, we have that $(1 \ot \pi_L(T)) Y_n - Y_n \pi_\cU(T) \to 0$ $*$-strongly for all $T \in M \otmax M\op$.

For all $\xi,\mu \in L^2(M) \ot L^2(\cU)$, we have that
\begin{multline*}
\langle (Y_n^* Y_n \pi_\cU(T) - \pi_\cU(T) Y_n^* Y_n) \xi,\mu \rangle \\
= \langle (Y_n \pi_\cU(T) - (1 \ot \pi_L(T)) Y_n) \xi, Y_n \mu \rangle + \langle Y_n \xi, ((1 \ot \pi_L(T^*)) Y_n - Y_n \pi_\cU(T^*))\mu \rangle \; ,
\end{multline*}
so that $Y_n^* Y_n \pi_\cU(T) - \pi_\cU(T) Y_n^* Y_n \to 0$ weakly for all $T \in M \otmax M\op$. It follows that $S$ commutes with $\pi_\cU(T)$, so that $S \in M \ovt B(L^2(G)) \cap \al(M)'$. By 1, we get that $S = 1 \ot F$, where $F \in L^\infty(\cU)$ is not essentially zero. We can then choose a compact subset $K \subset \cU$ and a $C > 0$ such that $\lambda(K) > 0$ and $F(x) \geq C^{-1}$ for all $x \in K$.

Define $Z_n = Y_n (1 \ot 1_K F^{-1/2})$. Note that $\|Z_n\| \leq C^{1/2}$ for all $n$. We still have that $(1 \ot \pi_L(T)) Z_n - Z_n \pi_K(T) \to 0$ $*$-strongly for all $T \in M \otmax M\op$. By construction, $Z_n^* Z_n \to 1 \ot 1_K$ weakly. For any $T \in M \otmax M\op$ and every $\xi \in L^2(M) \ot L^2(K,\lambda)$, we have
\begin{align*}
|\langle \pi_K(T) \xi,\xi \rangle &= |\langle (1 \ot 1_K) \pi_K(T) \xi,\xi\rangle| = \lim_n |\langle Z_n^* Z_n \pi_K(T) \xi,\xi \rangle| \\
&\leq \limsup_n C^{1/2} \, \|Z_n \pi_K(T) \xi\| \, \|\xi\| = \limsup_n C^{1/2} \, \|(1 \ot \pi_L(T)) Z_n \xi\| \, \|\xi\|\\
&\leq C \, \|\pi_L(T)\| \, \|\xi\|^2 \; .
\end{align*}
So, the homomorphism $\pi_L(T) \mapsto \pi_K(T)$ is well defined and continuous. We thus conclude that $\cH(K)$ is weakly contained $\cH(L)$. Choose $k \in K$ in the support of $\lambda|_K$, meaning that $\lambda(K \cap \cW) > 0$ for every open subset $\cW \subset G$ with $k \in \cW$. By Lemma \ref{lem.easy}, $\cH(\al_k)$ is weakly contained in $\cH(K)$ and is thus weakly contained in $\cH(L)$.

Since $L \subset G$ is a closed subset and $\pi \circ \al$ is a closed map, $\pi(\al(L))$ is a closed subset of $\Out M$ and $\pi(\al_k) \not\in \pi(\al(L))$. The weak containment of $\cH(\al_k)$ in $\cH(L)$ thus contradicts Lemma \ref{lem.amine}. This concludes the proof that $X_n$ converges to $0$ strongly.

{\bf Step B.} If $\cU \subset G$ is open and $\lambda(\overline{\cU} \setminus \cU) = 0$, then $a_n (1 \ot 1 \ot 1_\cU) - (1 \ot 1 \ot 1_\cU) a_n \to 0$ $*$-strongly.

Write $\cV = G \setminus \overline{\cU}$. Applying step~A to the open sets $\cU$ and $\cV$, we get that the sequences $(1 \ot 1 \ot 1_\cV) a_n (1 \ot 1 \ot 1_\cU)$ and $(1 \ot 1 \ot 1_\cU) a_n (1 \ot 1 \ot 1_\cV)$ converge to $0$ strongly. Since $1_\cV = 1- 1_\cU$ in $L^\infty(G)$, it follows that
$$a_n (1 \ot 1 \ot 1_\cU) - (1 \ot 1 \ot 1_\cU) a_n = (1 \ot 1 \ot 1_\cV) a_n (1 \ot 1 \ot 1_\cU) - (1 \ot 1 \ot 1_\cU) a_n (1 \ot 1 \ot 1_\cV) \to 0$$
strongly. The same holds for the sequence $a_n^*$, so that step~B is proven.

{\bf Step C.} For every bounded continuous function $F : G \to \C$, we have that $a_n (1 \ot 1 \ot F) - (1 \ot 1 \ot F) a_n \to 0$ $*$-strongly.

Fix a continuous function $F : G \to (0,1)$. It suffices to prove that $a_n (1 \ot 1 \ot F) - (1 \ot 1 \ot F) a_n \to 0$ $*$-strongly. Choose a probability measure $\lambda_1$ on $G$ such that $\lambda_1 \sim \lambda$. Define the probability measure $\mu$ on $\R$ by $\mu = F_*(\lambda_1)$. Fix an arbitrary $\eps > 0$. Since $\mu$ has at most countably many atoms, which all belong to $(0,1)$, we can choose $0 = t_0 < t_1 < \cdots < t_k = 1$ such that $\mu(\{t_i\}) = 0$ for all $0 \leq i \leq k$ and $t_i - t_{i-1} < \eps$ for all $1 \leq i \leq k$. Define, for $1 \leq i \leq k$, the open subsets
$$\cU_i = \{g \in G \mid t_{i-1} < F(g) < t_i \} \; .$$
Since $\mu(\{t_i\}) = 0$ for all $i$ and since $\lambda_1 \sim \lambda$, we get that $\lambda(\overline{\cU_i} \setminus \cU_i) = 0$ for all $1 \leq i \leq k$. Also, defining
$$F_0 = \sum_{i=1}^k t_i \, 1_{\cU_i} \; ,$$
we get that $\|F-F_0\| < \eps$. By step~B, $a_n (1 \ot 1 \ot F_0) - (1 \ot 1 \ot F_0)a_n \to 0$ $*$-strongly. So, for every $\eps > 0$, the sequence $a_n (1 \ot 1 \ot F) - (1 \ot 1 \ot F) a_n \to 0$ lies at operator norm distance less than $2 \eps$ of a sequence that converges to $0$ $*$-strongly. So, step~C is proven.

{\bf Step D.} Define the unitary $U \in B(L^2(M)) \ovt L^\infty(G)$ by $U(h) = U_h$ for all $h \in G$. Then, there exists a sequence $d_n \in B(H \ot L^2(M))$ such that $\|d_n\| \leq 1$ for all $n$ and $(1 \ot U) a_n (1 \ot U^*) - d_n \ot 1 \to 0$ $*$-strongly.

Write $b_n = (1 \ot U) a_n (1 \ot U^*)$. Let $F \in C_0(G)$. Since $U$ commutes with $1 \ot F$, it follows from step~C that $b_n (1 \ot 1 \ot F) - (1 \ot 1 \ot F) b_n \to 0$ $*$-strongly. Since $a_n \in B(H) \ovt (M \rtimes G)$, we have that $a_n$ commutes with $1 \ot U_g \ot \rho_g$ for every $g \in G$ and $n \in \N$. Thus, $b_n$ commutes with $1 \ot 1 \ot \rho_g$ for all $n \in \N$ and $g \in G$. Denoting by $C^*_\rho(G)$ the reduced C$^*$-algebra given by the right regular representation of $G$, we find that $b_n (1 \ot 1 \ot T) = (1 \ot 1 \ot T) b_n$ for all $T \in C^*_\rho(G)$. Since the operator norm closed linear span of $C_0(G) \, C^*_\rho(G)$ equals $\cK(L^2(G))$, it follows that $b_n (1 \ot 1 \ot T) - (1 \ot 1 \ot T) b_n \to 0$ $*$-strongly for all $T \in \cK(L^2(G))$. Step~D then follows from Lemma \ref{lem.compact}.

{\bf Step E.} With $U$ and $d_n$ as in step~D, the sequence $c_n \in B(H \ot L^2(M)) \ovt L^\infty(G)$ given by $c_n = (1 \ot U^*)(d_n \ot 1)(1 \ot U)$ satisfies $c_n (1 \ot b \ot 1) - (1 \ot b \ot 1) c_n \to 0$ $*$-strongly for all $b \in M$.

Since $a_n (1 \ot \al(b)) - (1 \ot \al(b)) a_n \to 0$ $*$-strongly and $\al(b) = U^* (b \ot 1)U$, the sequence $d_n$ satisfies $d_n (1 \ot b) - (1 \ot b) d_n \to 0$ $*$-strongly for all $b \in M$. Take $\mu \in H \ot L^2(M)$ and $\eta \in L^2(G)$ arbitrary. Then,
\begin{equation}\label{eq.good-sequence}
\begin{split}
\|(c_n (1 \ot b \ot 1) &- (1 \ot b \ot 1) c_n)(\mu \ot \eta)\|^2 \\ &= \int_G |\eta(g)|^2 \, \|((1 \ot U_g^*) d_n (1 \ot U_g \, b)  - (1 \ot b \, U_g^*) d_n (1 \ot U_g))\mu\|^2 \, dg \; .
\end{split}
\end{equation}
For every fixed $g \in G$ and $b \in M$, we have that $d_n (1 \ot \al_g(b)) - (1 \ot \al_g(b)) d_n \to 0$ strongly. Conjugating with $U_g^*$, also
$$(1 \ot U_g^*) d_n (1 \ot U_g \, b) - (1 \ot b \, U_g^*) d_n (1 \ot U_g) \to 0 \quad\text{strongly.}$$
It then follows from the dominated convergence theorem that the sequence in \eqref{eq.good-sequence} tends to $0$. We have proven that $c_n(1 \ot b \ot 1) - (1 \ot b \ot 1) c_n \to 0$ strongly for all $b \in M$. By symmetry, also $c_n^* (1 \ot b \ot 1) - (1 \ot b \ot 1) c_n^* \to 0$ strongly and step~E follows.

{\bf End of the proof.} By step~D, the sequence $c_n$ in step~E satisfies $a_n - c_n \to 0$ $*$-strongly. Since $a_n \in B(H) \ovt (M \rtimes G) \subset B(H) \ovt M \ovt B(L^2(G))$, we have that $a_n$ commutes with $1 \ot J b J \ot 1$ for all $b \in M$. Therefore, $c_n (1 \ot J b J \ot 1) - (1 \ot J b J \ot 1) c_n \to 0$ $*$-strongly for every $b \in M$. Since $M$ is a full factor, \cite[Theorem A]{Mar18b} says that the operator norm closed linear span of $M \, JMJ$ contains $\cK(L^2(M))$. The previous observation and step~E then imply that $c_n (1 \ot T \ot 1) - (1 \ot T \ot 1) c_n \to 0$ $*$-strongly for every $T \in \cK(L^2(M))$. Since $c_n \in B(H \ot L^2(M)) \ovt L^\infty(G)$, Lemma \ref{lem.compact} provides a sequence $F_n \in B(H) \ovt 1 \ovt L^\infty(G)$ such that $\|F_n\| \leq 1$ for all $n$ and $c_n - F_n \to 0$ $*$-strongly.

Then also $a_n - F_n \to 0$ $*$-strongly. Conjugating with $1 \ot U$, which commutes with $F_n$, and using step~D, we find that $d_n \ot 1 - F_n \to 0$ $*$-strongly. Fix a unit vector $\xi \in L^2(G)$ and denote by $\om_\xi$ the corresponding vector state on $B(L^2(G))$. Define the sequence $f_n \in B(H)$ such that $f_n \ot 1 = (\id \ot \id \ot \om_\xi)(F_n)$. Then, $\|f_n\| \leq 1$ for all $n \in \N$. Since $d_n \ot 1 - F_n \to 0$ $*$-strongly, also $d_n - f_n \ot 1 \to 0$ $*$-strongly. Conjugating with $1 \ot U^*$, it follows from step~D that $a_n - f_n \ot 1 \ot 1 \to 0$ $*$-strongly. So \ref{stat.c} holds.
\end{proof}

\section{Free Bogoljubov actions and proof of Corollary \ref{cor.free-araki-woods}}\label{sec.Araki-Woods}

Corollary \ref{cor.free-araki-woods} is an immediate consequence of Theorem \ref{thm.main} and the following result on free Araki-Woods factors. We expected that Proposition \ref{prop.free-Bog-Araki-Woods} would be well known and available in the literature, but this does not seem to be the case, not even in the case of the free Gaussian functor (i.e.\ Bogoljubov transformations of free group factors).

%{\bf ... (What do you think about the following? I would have thought that this was well known and written somewhere. But I could not find in the literature. And I noticed that in \cite{Hou12,HT18}, Cyril carefully avoids using this result. Even outerness of free Bogoljubov automorphisms seems to have only been proven in detail in \cite[Theorem 5.1]{HS09}. So it might be that Proposition \ref{prop.free-Bog-Araki-Woods} has never been written down. What do you think?) ...}

\begin{proposition}\label{prop.free-Bog-Araki-Woods}
Let $U : \R \actson H_\R$ be a strongly continuous action by orthogonal transformations of a separable real Hilbert space $H_\R$ with $\dim_\R H_\R \geq 2$. Denote by $M = \Gamma(H_\R,U)\dpr$ the associated free Araki-Woods factor, which is a full factor. Denote by $\pi : \Aut M \to \Out M$ the canonical quotient homomorphism.

Define the closed subgroup $\cG = \{v \in \cO(H_\R) \mid v U_t = U_t v \;\;\text{for all $t \in \R$}\,\}$ of $\cO(H_\R)$. For every $v \in \cG$, denote by $\al_v$ the associated Bogoljubov automorphism of $M$. Then, the homomorphism $\cG \to \Out M : v \mapsto \pi(\al_v)$ is injective, has closed range and is a homeomorphism onto this range.
\end{proposition}

Note that Proposition \ref{prop.free-Bog-Araki-Woods} covers in particular the case where $U_t = 1$ for all $t \in \R$. Then, $M = \Gamma(H_\R,U)\dpr \cong L(\F_{\dim_\R H_\R})$ and every $v \in \cO(H_\R)$ defines the free Bogoljubov automorphisms $\al_v$ of $M$. Proposition \ref{prop.free-Bog-Araki-Woods} then says that $\cO(H_\R) \to \Out M$ is injective, with closed range and is a homeomorphism onto this range.

Also note that, with the notation of Proposition \ref{prop.free-Bog-Araki-Woods}, the transformations $U_t$ belong to $\cG$. Therefore, Proposition \ref{prop.free-Bog-Araki-Woods} also reproves the result that Connes' $\tau$-invariant of a free Araki-Woods factor $\Gamma(H_\R,U)\dpr$ equals the weakest topology on $\R$ that makes the map $t \mapsto U_t$ continuous. This was proven in \cite[Corollary 8.6]{Shl97} when $U$ is not weakly mixing and in \cite[Th\'{e}or\`{e}me 2.7]{Vae04} in general.

Before proving Proposition \ref{prop.free-Bog-Araki-Woods}, we need a few elementary lemmas.

The first lemma is proven by Popa's spectral gap method. The proof is very similar to the proof of \cite[Lemma 2.2]{Pop06}. For completeness, we provide the details adapting the proof to a type III setting.

\begin{lemma}\label{lem.spectral-gap-free}
For $i \in \{1,2\}$, let $(M_i,\vphi_i)$ be von Neumann algebras with a faithful normal state, and with a separable predual. Define $(M,\vphi) = (M_1,\vphi_1) \ast (M_2,\vphi_2)$ and let $E_1 : M \to M_1$ be the canonical $\vphi$-preserving conditional expectation. Let $P \subset M_1$ be a von Neumann subalgebra that is the range of a faithful normal conditional expectation. Assume that $P$ has no amenable direct summand.

If $x_n \in M$ is a bounded sequence and $\al_n : P \to M_1$ is a sequence of faithful, unital, normal $*$-homomorphisms such that $x_n a - \al_n(a) x_n \to 0$ strongly for all $a \in P$, then $x_n - E_1(x_n) \to 0$ strongly.
\end{lemma}
\begin{proof}
For every von Neumann algebra $(Q,\om)$ with a faithful normal state, we denote by $L^2(Q,\om)$ the standard Hilbert space for $Q$ realized by completing $Q$ w.r.t.\ the scalar product $\langle x , y \rangle = \om(y^* x)$. As a $Q$-bimodule, we have $a \cdot b \cdot c = a b \sigma_{-i/2}^\om(c)$ for all $a,b,c \in M$ with $c$ being sufficiently analytic.

Fix a faithful normal conditional expectation $\cE : M_1 \to P$ and choose a faithful normal state $\psi_1$ on $M_1$ satisfying $\psi_1 \circ \cE = \psi_1$. Define the faithful normal state $\psi$ on $M$ by $\psi = \psi_1 \circ E_1$. Denote $M \ominus M_1 = \{x \in M \mid E_1(x) = 0\}$. We also write $M_i^\circ = \{a \in M_i \mid \vphi_i(a) = 0\}$. Whenever $w \in M$ is defined as an alternating product of elements in $M_1^\circ$ and $M_2^\circ$, starting and ending with a `letter' from $M_2^\circ$, we have, for all $a,b \in M_1$,
$$E_1(w^* a w b) = \vphi_1(a) \, \vphi(w^*w) \, b \quad\text{so that}\quad \langle a w b, w \rangle_\psi = \psi(w^* a w b) = \vphi_1(a) \, \psi_1(b) \, \vphi(w^* w) \; .$$
It follows that there exists an $M_1$-bimodular isometry
$$U : L^2(M \ominus M_1,\psi) \to K = (L^2(M_1,\vphi_1) \ot L^2(M_1,\psi_1))^{\oplus \infty} \; ,$$
where the $M_1$-bimodule structure on $K$ is given by the left action in the first tensor factor $L^2(M_1,\vphi_1)$ and the right action in the second tensor factor $L^2(M_1,\psi_1)$.

Write $y_n = x_n - E_1(x_n) \in M \ominus M_1$. Note that $y_n$ is still a bounded sequence in $M$ satisfying $y_n a - \al_n(a) y_n \to 0$ strongly for all $a \in P$. Define the bounded linear maps
$$Y_n : L^2(P,\psi_1) \to K : Y_n(a) = U(y_n a) \; .$$
Note that $\sup_n \|Y_n\| \leq \sup_n \|y_n\| < +\infty$. Take $\kappa > 0$ such that $\|Y_n\| \leq \kappa$ for all $n \in \N$. By construction, $Y_n(\xi \cdot a) = Y_n(\xi) \cdot a$ for all $\xi \in L^2(P,\psi_1)$ and $a \in P$. Also by construction, $\|\al_n(a) \cdot Y_n(b) - Y_n(a \cdot b)\| \to 0$ for all $a,b \in P$. Given $a \in P$, the sequences $\|\al_n(a)\|$ and $\|Y_n\|$ are bounded, so that $\|\al_n(a) \cdot Y_n(\xi) - Y_n(a \cdot \xi)\| \to 0$ for all $a \in P$ and $\xi \in L^2(P,\psi_1)$.

Define the $*$-representations $\pi_n : P \otalg P\op \to B(K)$ by $\pi_n(a \ot b\op) \xi = \al_n(a) \cdot \xi \cdot b$. Note that $\|\pi_n(T)\| = \|\lambda(T)\|$ for every $T \in P \otalg P\op$, where $\lambda$ is given by the coarse $P$-bimodule. Also define $\eps : P \otalg P\op \to B(L^2(P,\psi_1)) : \eps(a \ot b\op)\xi = a \cdot \xi \cdot b$. We have proven that
\begin{equation}\label{eq.good}
\pi_n(T) Y_n - Y_n \eps(T) \to 0 \quad\text{strongly for every $T \in P \otalg P\op$.}
\end{equation}
We claim that $Y_n \to 0$ strongly. Let $S \in B(L^2(P,\psi_1))^+$ be any weak limit point of the bounded sequence $Y_n^* Y_n$. Since
$$\langle (Y_n^* Y_n \eps(T) - \eps(T) Y_n^* Y_n) \xi,\mu \rangle = \langle (Y_n \eps(T) - \pi_n(T) Y_n)\xi,Y_n\mu \rangle + \langle Y_n \xi, (\pi_n(T^*) Y_n - Y_n \eps(T^*)) \mu \rangle \; ,$$
it follows from \eqref{eq.good} that $S$ commutes with $\eps(T)$ for all $T \in P \otalg P\op$. Thus, $S \in \cZ(P)$. For all $\xi,\eta \in L^2(P,\psi_1)$ and $T \in P \otalg P\op$, we get that
\begin{align*}
|\langle S \eps(T) \xi,\eta \rangle| &\leq \limsup_n |\langle Y_n^* Y_n \eps(T) \xi, \eta \rangle| = \limsup_n |\langle Y_n \eps(T) \xi, Y_n \eta \rangle| \\
& \leq \kappa \, \|\eta\| \, \limsup_n \|Y_n \eps(T) \xi\| = \kappa \, \|\eta\| \, \limsup_n \|\pi_n(T) Y_n \xi\| \leq \kappa^2 \, \|\xi\| \, \|\eta\| \, \|\lambda(T)\| \; .
\end{align*}
We conclude that $\|S \eps(T)\| \leq \kappa^2 \, \|\lambda(T)\|$ for all $T \in P \otalg P\op$. If $S$ is nonzero, it follows that $P$ has an amenable direct summand. So, $S = 0$ and the claim that $Y_n \to 0$ strongly is proven.

Since $y_n = Y_n(1)$, we have proven that $\|y_n\|_{2,\psi} \to 0$. Thus, $y_n \to 0$ strongly.
\end{proof}

\begin{proposition}\label{prop.aut-free-product}
For $i \in \{1,2\}$, let $(M_i,\vphi_i)$ be a von Neumann algebra with a faithful normal state and separable predual. Assume that $(M_1,\vphi_1)$ has no amenable direct summand and that $M_2 \neq \C 1$. Denote by $\Aut(M_i,\vphi_i)$ the Polish group of state preserving automorphisms.

Define $(M,\vphi) = (M_1,\vphi_1) \ast (M_2,\vphi_2)$. Then $M$ is a full factor. Denote by $\pi : \Aut M \to \Out M$ the natural quotient homomorphism. Then, the homomorphism
$$\theta : \Aut(M_1,\vphi_1) \times \Aut(M_2,\vphi_2) \to \Out M : \theta(\al,\be) = \pi(\al \ast \be)$$
is faithful, has closed range and is a homeomorphism onto this range.
\end{proposition}
\begin{proof}
It suffices to prove the following statement: if $\al_n \in \Aut(M_1,\vphi_1)$ and $\be_n \in \Aut(M_2,\vphi_2)$ are sequences of state preserving automorphisms such that $\Ad u_n^* \circ (\al_n \ast \be_n) \to \id$ for some sequence of unitaries $u_n \in M$, then $\al_n \to \id$, $\be_n \to \id$ and $u_n - \vphi(u_n) 1 \to 0$ $*$-strongly. By taking $\al_n = \be_n = \id$, this then says in particular that $M$ is full.

Denote by $E_i : M \to M_i$ the canonical $\vphi$-preserving conditional expectation. For every $a \in M_1$, we have that $u_n a - \al_n(a) u_n \to 0$ strongly. It follows from Lemma \ref{lem.spectral-gap-free} that $u_n - E_1(u_n) \to 0$ strongly.

Write $v_n = (\al_n \ast \be_n)^{-1}(u_n^*)$. Then also $\Ad v_n^* \circ (\al_n^{-1} \ast \be_n^{-1}) \to \id$. Applying the previous paragraph, it follows that $\|v_n - E_1(v_n)\|_{2,\vphi} \to 0$. Since $\al_n \ast \be_n$ is state preserving and commutes with $E_1$, we conclude that $\|u_n^* - E_1(u_n^*)\|_{2,\vphi} \to 0$, so that $u_n^* - E_1(u_n^*) \to 0$ strongly. In combination with the previous paragraph, we have proven that $u_n - E_1(u_n) \to 0$ $*$-strongly.

Write $x_n = E_1(u_n)$. Note that $\|x_n\| \leq 1$ for all $n$. Since $M_2 \neq \C 1$, we can fix $b \in M_2$ with $\vphi_2(b) = 0$ and $\vphi_2(b^* b) = 1$. Since $u_n^* \be_n(b) u_n \to b$ weakly and $u_n - x_n \to 0$ strongly, we get that $x_n^* \be_n(b) x_n \to b$ weakly and thus $E_2(x_n^* \be_n(b) x_n) \to b$ weakly. The left hand side equals $|\vphi_1(x_n)|^2 \, \be_n(b)$. Therefore,
$$\langle |\vphi_1(x_n)|^2 \, \be_n(b) , b \rangle = \langle E_2(x_n^* \be_n(b) x_n), b \rangle \to \langle b, b \rangle = 1 \; .$$
Since $|\vphi_1(x_n)| \leq 1$ and $|\langle \be_n(b), b \rangle| \leq 1$ because $\be_n$ is state preserving, we conclude that $|\vphi_1(x_n)| \to 1$. Since $\|x_n\| \leq 1$ for all $n$, this implies that $x_n - \vphi_1(x_n) 1 \to 0$ $*$-strongly. It thus follows that $u_n - \vphi(u_n) 1 \to 0$ $*$-strongly.

Then also $\Ad u_n \to \id$, so that $\al_n \ast \be_n \to \id$. By restricting to $M_i$, it follows that $\al_n \to \id$ and $\be_n \to \id$.
\end{proof}

\begin{lemma}\label{lem.playing-with-orth}
Let $H_\R$ be a separable infinite dimensional real Hilbert space. Let $V_n \in \cO(H_\R)$ such that $(V_n)_{n \in \N}$ is weakly convergent. Let $N \in \N$.

There exist orthonormal vectors $\{e_1,\ldots,e_N\}$, orthonormal vectors $\{f_1,\ldots,f_N\}$ and $W_n \in \cO(H_\R)$ such that
\begin{equation}\label{eq.good-limits}
\begin{split}
& W_n e_i = e_i \;\;\text{and}\;\; W_n V_n e_i \to f_i \;\;\text{in norm for every $i \in \{1,\ldots,N\}$, and}\\
& (\R e_i + \R f_i) \perp (\R e_j + \R f_j) \;\;\text{for all $i,j \in \{1,\ldots,N\}$ with $i \neq j$.}
\end{split}
\end{equation}
An analogous result holds for a weakly convergent sequence of unitaries on a separable infinite dimensional Hilbert space.
\end{lemma}
\begin{proof}
Let $V_n \to T$ weakly. Note that $\|T\| \leq 1$. Inductively choose unit vectors $\{e_1,\ldots,e_N\}$ such that whenever $i < j$, we have that $e_j \perp \{e_i, Te_i, T^* e_i, T^* T e_i\}$. Define $\mu_{k,n} = V_n e_k - T e_k$ and define $\al_k \geq 0$ such that $\al_k^2 = 1 - \|T e_k\|^2$. By construction,
$$V_n e_k = T e_k + \mu_{k,n} \;\; , \quad \mu_{k,n} \to 0 \;\;\text{weakly, and}\quad \|\mu_{k,n}\| \to \al_k \; .$$
When $i \neq j$, we have that $V_n e_i \perp V_n e_j$. Also, $T e_i \perp T e_j$ by construction. Therefore, for $i \neq j$, we have
$$0 = \langle V_n e_i, V_n e_j \rangle = \langle T e_i , \mu_{j,n} \rangle + \langle \mu_{i,n}, T e_j \rangle + \langle \mu_{i,n}, \mu_{j,n} \rangle \; .$$
Since $\mu_{k,n} \to 0$ weakly, it follows that $\langle \mu_{i,n}, \mu_{j,n}\rangle \to 0$.

Define $K$ as the linear span of all the vectors $e_i$ and $T e_i$ with $i \in \{1,\ldots,N\}$. Let $L = K^\perp$. Since $\mu_{k,n} \to 0$ weakly, $\|P_K \mu_{k,n}\| \to 0$. Put $\eta_{k,n} = P_L \mu_{k,n}$, so that $\|\mu_{k,n} - \eta_{k,n}\| \to 0$. It follows in particular that $\|\eta_{k,n}\| \to \al_k$ and that $\langle \eta_{i,n},\eta_{j,n}\rangle \to 0$ when $i \neq j$.

Choose orthonormal vectors $\{g_1,\ldots,g_N\}$ in $L$. By the Gram-Schmidt procedure, we can choose $W_n \in \cO(H_\R)$ such that $W_n \xi = \xi$ for all $\xi \in K$ and $W_n \eta_{k,n} \to \al_k g_k$ in norm. By construction,
$$W_n V_n e_k \to T e_k + \al_k g_k \;\;\text{in norm.}$$
Define $f_k = T e_k + \al_k g_k$. By construction, the conclusion of the lemma holds.
\end{proof}

\begin{proof}[{Proof of Proposition \ref{prop.free-Bog-Araki-Woods}}]
When $H_\R$ is finite dimensional, the group $\cG$ is compact. By \cite[Lemma 3.2]{HT18}, every $\al_v$ with $v \in \cO(H_\R) \setminus \{1\}$ is an outer automorphism of $M$ and the proposition follows. We thus assume that $H_\R$ is infinite dimensional.

The spectral measure class of the orthogonal representation $(U_t)_{t \in \R}$ is a symmetric measure class on $\R$ that we denote by $\mu$.

{\bf Case 1~:} there is no $a \geq 0$ such that the measure class $\mu$ is concentrated on $\{-a,a\}$. We can then choose a Borel set $\cU_1 \subset \R$ such that $\cU_1 = - \cU_1$ and such that both $\cU_1$ and its complement $\cU_2 = \R \setminus \cU_1$ have positive measure. Make this choice such that $0 \in \cU_2$. Denote by $H^i_\R \subset H_\R$ the spectral subspace corresponding to $\cU_i$. We have $H_\R = H^1_\R \oplus H^2_\R$ and $\cG = \cG_1 \times \cG_2$. Then the free Araki-Woods factor $(M,\vphi) = \Gamma(H_\R,U)\dpr$ is the free product of the free Araki-Woods factors $(M_i,\vphi_i) = \Gamma(H^i_\R,U)\dpr$. Since $0 \not\in \cU_1$, we have that $\dim_\R H^1_\R \geq 2$, so that $M_1$ has no amenable direct summand. Since $\cG = \cG_1 \times \cG_2$, we identify $\cG$ with a closed subgroup of $\Aut(M_1,\vphi_1) \times \Aut(M_2,\vphi_2)$. Proposition \ref{prop.free-Bog-Araki-Woods} then follows from Proposition \ref{prop.aut-free-product}.

{\bf Case 2~:} the measure class $\mu$ is concentrated on $\{-a,a\}$ for some $a \geq 0$. Let $V_n \in \cG$ be a sequence such that $\pi(\al_{V_n}) \to \id$. We have to prove that $V_n \to 1$ weakly. Passing to a subsequence, we may assume that $V_n \to T$ weakly and we have to prove that $T=1$. Take unitaries $u_n \in \cU(M)$ such that $\Ad u_n^* \circ \al_{V_n} \to \id$.
%Since $V_n \in \cO(H_\R)$, it suffices to prove that $V_n \to 1$ weakly. Let $T \in B(H_\R)$ be any weak limit point of the sequence $(V_n)_{n \in \R}$. We have to prove that $T=1$. After passage to a subsequence, we may assume that $V_n \to T$ weakly.

We claim that there exist $2$-dimensional real subspaces $K^1_\R, K^2_\R \subset H_\R$ and a sequence $W_n \in \cG$ with the following properties: the subspaces $K^j_\R$ are globally invariant under $(U_t)_{t \in \R}$, we have $K^1_\R \perp K^2_\R$, writing $K_\R = K^1_\R + K^2_\R$, we have that $W_n \xi = \xi$ for all $\xi \in K_\R$ and $n \in \N$, and
$$W_n V_n \xi \to F \xi \quad\text{in norm, for all $\xi \in K_\R$,}$$
where $F : K_\R \to H_\R$ is an isometry that commutes with $(U_t)_{t \in \R}$ and satisfies $(K^1_\R + F(K^1_\R)) \perp (K^2_\R + F(K^2_\R))$.

When $a = 0$, we have that $U_t = 1$ for all $t \in \R$. In that case, the claim follows from Lemma \ref{lem.playing-with-orth}, by taking $K^1_\R = \R e_1 + \R e_2$ and $K^2_\R = \R e_3 + \R e_4$.

When $a > 0$, write $\lambda = \exp(a)$. We can then identify the complexification $H_\R + i H_\R$ with a Hilbert space of the form $\cH \oplus \overline{\cH}$, where $\cH$ is a (complex) Hilbert space, $U_t(\xi,\overline{\eta}) = (\lambda^{it} \xi, \lambda^{-it} \overline{\eta})$ and $H_\R = \{(\xi,\overline{\xi}) \mid \xi \in \cH\}$. Then, $\cG$ is identified with $\cU(\cH)$, with every unitary $V \in \cU(\cH)$ giving rise to the orthogonal transformation of $H_\R$ defined by restricting $V \oplus \overline{V}$ to $H_\R$. So, we view $V_n \in \cG$ as the sequence $V_n \oplus \overline{V_n}$.

By the complex version of Lemma \ref{lem.playing-with-orth}, we can choose orthonormal vectors $\{e_1,e_2\}$ and $\{f_1,f_2\}$ in $\cH$, and unitaries $W_n \in \cU(\cH)$ such that $W_n e_i = e_i$ for all $n \in \N$ and $i \in \{1,2\}$, while $W_n V_n e_i \to f_i$ in norm. Also, $(\C e_1 + \C f_1) \perp (\C e_2 + \C f_2)$. We can now define $K^i_\R \subset H_\R$ by
$$K^i_\R = \{(z e_i , \overline{z e_i}) \mid z \in \C\}$$
and use the orthogonal transformations $W_n \oplus \overline{W_n}$. Again, the claim is proven.

Write $L^i_\R = K^i_\R + F(K^i_\R)$. Define $L^0_\R = (L^1_\R + L^2_\R)^\perp$. We find the free product decomposition $(M,\vphi) = (M_0,\vphi_0) \ast (M_1,\vphi_1) \ast (M_2,\vphi_2)$, where $M_i = \Gamma(L^i_\R,U)\dpr$. Denote by $E_i : M \to M_i$ the canonical $\vphi$-preserving conditional expectation. For $i \in \{1,2\}$, we also have the subalgebras $P_i \subset M_i$ given by $P_i = \Gamma(K^i_\R,U)\dpr$. Note that there exists a faithful normal conditional expectation of $M_i$ onto $P_i$. Since $K^i_\R$ is $2$-dimensional, the von Neumann algebra $P_i$ is a nonamenable factor.

The restriction of $F$ to $K^i_\R$ gives rise to the state preserving embedding $\be_i : P_i \to M_i$. Write $v_n = \al_{W_n}(u_n)$. Since $\Ad u_n^* \circ \al_{V_n} \to \id$, we get that
$$\|\al_{V_n}(a) u_n - u_n a\|_{2,\vphi} \to 0 \quad\text{for all $a \in M$.}$$
Since $\al_{W_n}$ is state preserving and $W_n \xi = \xi$ for all $\xi \in K^i_\R$, it follows that
$$\|\al_{W_n V_n}(a) v_n - v_n a \|_{2,\vphi} \to 0 \quad\text{for all $a \in P_i$.}$$
Since $W_n V_n \xi \to F \xi$ for all $\xi \in K^i_\R$, we get that
\begin{equation}\label{eq.without}
\|\al_{W_n V_n}(a) - \be_i(a)\|_{2,\vphi} \to 0 \quad\text{for all $a \in P_i$.}
\end{equation}
Since $\|\vphi \circ \Ad u_n^* \circ \al_{V_n} - \vphi\| \to 0$ and since $\vphi$ is invariant under $\al_{W_n}$ and $\al_{W_n V_n}$, we get that $\|\vphi \circ \Ad v_n^* - \vphi \|_{2,\vphi} \to 0$. It then follows from \eqref{eq.without} that $\|\al_{W_n V_n}(a)v_n - \be_i(a)v_n\|_{2,\vphi} \to 0$ for all $a \in P_i$. We thus conclude that
$$\|\be_i(a) v_n - v_n a\|_{2,\vphi} \to 0 \quad\text{for all $a \in P_i$.}$$
By Lemma \ref{lem.spectral-gap-free}, we get that $\|v_n - E_i(v_n)\|_{2,\vphi} \to 0$. Since $E_1(E_2(x)) = \vphi(x)1$ for all $x \in M$, we conclude that $\|v_n - \vphi(v_n)1\|_{2,\vphi} \to 0$. Since $\al_{W_n}$ is state preserving, it follows that $\|u_n - \vphi(u_n)1\|_{2,\vphi} \to 0$. In particular, $|\vphi(u_n)| \to 1$. This means that $u_n - \vphi(u_n)1 \to 0$ $*$-strongly. Thus, $\Ad u_n \to \id$. Then also $\al_{V_n} \to \id$, from which it follows that $V_n \to 1$ strongly, so that $T = 1$.
\end{proof}

\section{Questions, comments and counterexamples}\label{final.sec}

We start by providing the following easy examples of outer actions that are not strictly outer.

\begin{example}\label{exam.counterexample}
First, if $M$ is a factor of type III$_0$ whose $T$-invariant is trivial and if $\vphi$ is a faithful normal state on $M$, the action $\R \actson^{\sigma^\vphi} M$ by modular automorphisms is outer, but the crossed product $M \rtimes_{\sigma^\vphi} \R$ is not even a factor. In particular, $\si^\vphi$ is not strictly outer.

Second, if $\Gamma$ is a countable abelian group and $\Gamma \actson^\be (X,\mu)$ is an essentially free, probability measure preserving, weakly mixing action, then the dual action $\al$ of $K = \widehat{\Gamma}$ on the crossed product II$_1$ factor $M = L^\infty(X) \rtimes_\be \Gamma$ is outer, but again, the crossed product $M \rtimes_\al K \cong L^\infty(X) \ovt B(\ell^2(\Gamma))$ is not even a factor.
\end{example}

Note that in Example \ref{exam.counterexample}, the actions on $M$ are by approximately inner automorphisms: in \cite[Proposition 3.9]{Con74} (see also \cite[Corollary 6.24]{MT12}), it was proven that every modular automorphism of a type III$_0$ factor is approximately inner, while in the second example, it follows from the Rohlin lemma (see \cite{OW79}) that for every $\om \in K$, there exists a sequence of unitaries $u_n \in \cU(L^\infty(X,\mu))$ such that $\be_g(u_n) - \om(g) u_n \to 0$ $*$-strongly for every $g \in \Gamma$. Then, $\Ad u_n^* \to \al_\om$ in $\Aut M$.

In combination with Theorem \ref{thm.main}, this thus leads naturally to the following question, asking whether actions by not approximately inner automorphisms are automatically strictly outer.

\begin{question}\label{question.strict-outer}
Let $G \actson^\al M$ be a strongly continuous action of a lcsc group $G$ on a factor $M$ with separable predual. Denote by $\pi : \Aut M \to \Aut M / \overline{\Inn M}$ the natural quotient homomorphism.

Does the faithfulness of $\pi \circ \al$ imply that $\al$ is strictly outer?
\end{question}

One should not expect to give an easy proof for a positive answer to question \ref{question.strict-outer}. Indeed, when $M$ is a II$_\infty$ factor, $G = \R$ and $\R \actson^\al M$ is a trace scaling action, the relative commutant theorem \cite[Theorem 5.1]{CT76} says that the action $\al$ is strictly outer. Since $\al$ is trace scaling, it is trivial that the homomorphism from $\R$ to $\Aut M / \overline{\Inn M}$ is faithful. Therefore, a positive answer to question \ref{question.strict-outer} also has to cover the notoriously difficult relative commutant theorem in modular theory.

In the very specific case of a tensor product action with one of the tensor factors being full, the following result provides a positive answer to question \ref{question.strict-outer}.

%As we said in the introduction, it is very natural to try to generalize Theorem \ref{thm.main} to factors $M$ that are not full, replacing $\Out M$ in the hypotheses by the Polish group $\Aut M / \overline{\Inn M}$. We prove such a generalization for specific tensor product actions in Propositions \ref{prop.product-strictly-outer} and \ref{prop.product-fullness} and pose the problem explicitly as Questions \ref{question.strict-outer} and \ref{question.fullness}.

\begin{proposition}\label{prop.product-strictly-outer}
Let $G \actson^\al N$ and $G \actson^\be M$ be strongly continuous actions of a locally compact group $G$ on factors $N$ and $M$. Consider the diagonal action $\gamma_g = \al_g \ot \be_g$ of $G$ on $N \ovt M$.

If $\be$ is strictly outer, then $\gamma$ is strictly outer.
\end{proposition}

%It thus follows from Proposition \ref{prop.product-strictly-outer} and Theorem \ref{thm.main} that a diagonal action $\gamma = \al \ot \be$ is outer when $\be$ is an outer action on a full factor $M$ and $\al$ is an arbitrary action on a factor $N$. This provides some evidence for a positive answer to the following question.

\begin{proof}%[{Proof of Proposition \ref{prop.product-strictly-outer}}]
Define $\gamma : N \ovt M \to N \ovt M \ovt L^\infty(G) : \gamma(d)(g) = \gamma_{g^{-1}}(d)$, so that $(N \ovt M) \rtimes_\gamma G$ is the von Neumann subalgebra of $N \ovt M \ovt B(L^2(G))$ generated by $\gamma(N \ovt M)$ and $1 \ot 1 \ot L(G)$. We similarly define $\be : M \to M \ovt L^\infty(G)$ so that $M \rtimes_\be G$ is the von Neumann subalgebra of $M \ovt B(L^2(G))$ generated by $\be(M)$ and $1 \ot L(G)$.

Denote by $N \subset B(H)$ the standard representation of $N$ and let $(U_g)_{g \in G}$ be the canonical implementation of the action $\al$. Define the unitary $U \in B(H) \ovt 1 \ovt L^\infty(G)$ by $U(g) = U_g \ot 1$. Then,
\begin{equation}\label{eq.theta}
\theta : (N \ovt M)\rtimes_\gamma G \to B(H) \ovt (M \rtimes_\be G) : \theta(T) = U T U^*
\end{equation}
is a well defined $*$-homomorphism satisfying
$$\theta(\gamma(a \ot b)) = a \ot \be(b) \quad\text{and}\quad \theta(1 \ot 1 \ot \lambda_g) = U_g \ot 1 \ot \lambda_g \quad\text{for all $a \in N$, $b \in M$, $g \in G$.}$$
Let $a \in \gamma(N \ovt M)' \cap (N \ovt M) \rtimes_\gamma G$. Then, $\theta(a)$ is an element of $B(H) \ovt (M \rtimes_\be G)$ that commutes with $1 \ot \be(M)$. Since $\be$ is strictly outer, we conclude that $\theta(a) = b \ot 1 \ot 1$ for some $b \in B(H)$.

Since $\theta(a)$ commutes with $N \ot 1 \ot 1$, we get that $b = J c J$ with $c \in N$. Since $a$ belongs to $N \ovt M \ovt B(L^2(G))$, we get that $a$ commutes with $J d J \ot 1 \ot 1$ for every $d \in N$. Note that $U(JdJ \ot 1 \ot 1)U^*$ belongs to $B(H) \ovt 1 \ovt L^\infty(G)$ and is given by the function $g \mapsto J \al_g(d) J \ot 1$. Since $\theta(a) = J c J \ot 1 \ot 1$ commutes with all these operators, we conclude that $c \in \cZ(N) = \C 1$. Thus, $a \in \C 1$.
\end{proof}

In Theorem \ref{thm.main}, we have proven that crossed products $M \rtimes_\al G$ are full whenever $M$ is full and the homomorphism $G \to \Out M$ is faithful and has closed range. In light of question \ref{question.strict-outer}, it is natural to try to prove a similar result without assuming that $M$ is full, under the hypothesis that the homomorphism $G \to \Aut M / \overline{\Inn M}$ is faithful and has closed range. One may expect that this condition should be sufficient to prove that centralizing sequences of $M \rtimes_\al G$ (see below for terminology) are equivalent with asymptotically $G$-invariant centralizing sequences in $M$. This leads to the following question, for which we again provide positive evidence in the case of a product action (see Proposition \ref{prop.product-fullness}).

Recall that a \emph{centralizing} sequence in a von Neumann algebra $N$ with separable predual is a sequence $a_n \in N$ satisfying $\sup_n \|a_n\| < +\infty$ and $\| \om \cdot a_n - a_n \cdot \om\| \to 0$ for every $\om \in N_*$. A centralizing sequence $(a_n)_{n \in \N}$ is said to be \emph{trivial} if there exists a bounded sequence $t_n \in \C$ such that $a_n - t_n 1 \to 0$ $*$-strongly. Note that a factor $N$ with separable predual is full if and only if all centralizing sequences in $N$ are trivial.

\begin{question}\label{question.fullness}
Let $G \actson^\al M$ be a strongly continuous action of a lcsc group $G$ on a factor $M$ with separable predual. Denote by $\pi : \Aut M \to \Aut M / \overline{\Inn M}$ the natural quotient homomorphism. Make the following assumptions.
\begin{itemlist}
\item The homomorphism $\pi \circ \al$ is faithful and has closed image.
\item Every centralizing sequence $a_n \in M$ satisfying $\al_g(a_n) - a_n \to 0$ $*$-strongly for every $g \in G$, is trivial.
\end{itemlist}
Does it follow that $M \rtimes G$ is a full factor?
\end{question}

Question \ref{question.fullness} is much more challenging than question \ref{question.strict-outer}. While question \ref{question.strict-outer} is trivial for discrete groups $G$ (because then every outer action is strictly outer), question \ref{question.fullness} is even open when $G$ is a discrete group and $M$ is a II$_1$ factor. The problem is that we have no analogue of Lemma \ref{lem.amine} when $M$ is no longer full. The only positive result in this direction is provided by \cite[Theorem 3.1]{Con75} saying that for a II$_1$ factor $M$ and a single automorphism $\al \in \Aut M$, we have that the trivial bimodule $L^2(M)$ is weakly contained in $\cH(\al)$ if and only if $\al$ is approximately inner. Even for a countable family of automorphisms, this is not clear: assume that $\al_n \in \Aut M$, $n \in \N$, is a sequence of automorphisms of a II$_1$ factor $M$. Assume that the identity $\id$ does not belong to the closure of $\{\pi(\al_n) \mid n \in \N\}$, where $\pi : \Aut M \to \Aut M / \overline{\Inn M}$ is the quotient homomorphism. Can we conclude that the trivial bimodule $L^2(M)$ is not weakly contained in $\bigoplus_{n \in \N} \cH(\al_n)$~? When $M$ is of type II$_\infty$ or type III, the situation is even more delicate, since then Lemma \ref{lem.amine} is no longer true. For instance, when $\al$ is a trace scaling automorphism of the hyperfinite II$_\infty$ factor $M$, the bimodules $L^2(M)$ and $\cH(\al)$ are weakly equivalent, but $\al$ is not approximately inner.

As in Proposition \ref{prop.product-strictly-outer}, we give a positive answer to question \ref{question.fullness} in the case of product actions.

\begin{proposition}\label{prop.product-fullness}
Let $G \actson^\al N$ and $G \actson^\be M$ be strongly continuous actions of a lcsc group $G$ on factors $N$ and $M$ with separable predual. Denote by $\gamma_g = \al_g \ot \be_g$ the diagonal action of $G$ on $N \ovt M$. Consider the following two statements.
\begin{enumlist}
\item $(N \ovt M) \rtimes_\gamma G$ is a full factor.
\item Every centralizing sequence $a_n \in N$ satisfying $\al_g(a_n) - a_n \to 0$ $*$-strongly for every $g \in G$, is trivial.
\end{enumlist}
Then, $1 \Rightarrow 2$. If $\be$ is strictly outer and that $M \subset M \rtimes_\be G$ has stable $w$-spectral gap, then also $2 \Rightarrow 1$.
\end{proposition}

From Theorem \ref{thm.main}, it thus follows that the two statements in Proposition \ref{prop.product-fullness} are equivalent when $G \actson^\al N$ is an arbitrary action and $G \actson^\be M$ is an action on a full factor for which the resulting homomorphism $G \to \Out M$ is faithful and has closed image.
%In the same way as Proposition \ref{prop.product-strictly-outer} and Theorem \ref{thm.main} provided some evidence for a positive answer to question \ref{question.strict-outer}, we thus find some evidence for a positive answer to the following even more challenging question.

\begin{proof}%[{Proof of Proposition \ref{prop.product-fullness}}]
We use the same notation as in the beginning of the proof of Proposition \ref{prop.product-fullness}.

We first prove that $1 \Rightarrow 2$. Take a centralizing sequence $a_n \in N$ satisfying $\al_g(a_n) - a_n \to 0$ $*$-strongly for every $g \in G$. It follows from the dominated convergence theorem that $\gamma(a_n \ot 1) - a_n \ot 1 \ot 1 \to 0$ $*$-strongly. We get for every $\om \in (N \ovt M \ovt B(L^2(G)))_*$ w.r.t.\ the norm of $(N \ovt M \ovt B(L^2(G)))_*$ that
\begin{align*}
& \| \om \cdot \gamma(a_n \ot 1) - \om \cdot (a_n \ot 1 \ot 1)\| \to 0 \;\; , \quad \| \gamma(a_n \ot 1) \cdot \om - (a_n \ot 1 \ot 1) \cdot \om \| \to 0 \quad\text{and}\\
& \| \om \cdot (a_n \ot 1 \ot 1)  - (a_n \ot 1 \ot 1) \cdot \om\| \to 0 \; .
\end{align*}
It follows that $\|\om \cdot \gamma(a_n \ot 1) - \gamma(a_n \ot 1) \cdot \om\| \to 0$ for all $\om \in (N \ovt M \ovt B(L^2(G)))_*$. A fortiori, $(\gamma(a_n \ot 1))_{n \in \N}$ is a centralizing sequence in $(N \ovt M) \rtimes_\gamma G$. By 1, the sequence $(\gamma(a_n \ot 1))_{n \in \N}$ is trivial, so that also the centralizing sequence $a_n \in N$ must be trivial.

We then prove the more subtle converse implication $2 \Rightarrow 1$, assuming that $\be$ is strictly outer and that $\be(M) \subset M \rtimes_\be G$ has stable $w$-spectral gap. Fix a centralizing sequence $a_n \in (N \ovt M) \rtimes_\gamma G$. We may assume that $\|a_n\| \leq 1$ for all $n$. We use the homomorphism $\theta$ defined in \eqref{eq.theta}. Since $(a_n)_{n \in \N}$ is centralizing, we get that $\theta(a_n) (1 \ot \be(b)) - (1 \ot \be(b)) \theta(a_n) \to 0$ $*$-strongly for every $b \in M$. Since $\be(M) \subset M \rtimes_\be G$ has stable $w$-spectral gap and trivial relative commutant, we find a sequence $b_n \in B(H)$ such that $\|b_n\| \leq 1$ for all $n$ and $\theta(a_n) - b_n \ot 1 \ot 1 \to 0$ $*$-strongly. It follows that $a_n - U^*(b_n \ot 1 \ot 1)U \to 0$ $*$-strongly.

Denote by $\lambda$ a left Haar measure on $G$. Fix a Borel set $\cV \subset G$ with $0 < \lambda(\cV) < +\infty$. Denote by $\om_\cV$ the vector state on $B(L^2(G))$ given by the unit vector $\lambda(\cV)^{-1/2} 1_\cV$. Fix a normal state $\om \in M_*$. Since $a_n \in N \ovt M \ovt B(L^2(G))$, define the sequence $c_n \in N$ by $c_n = (\id \ot \om \ot \om_\cV)(a_n)$. Since $a_n - U^*(b_n \ot 1 \ot 1)U \to 0$ $*$-strongly, we define
$$d_n = \lambda(\cV)^{-1} \int_\cV U_g^* b_n U_g \; dg$$
and conclude that $c_n - d_n \to 0$ $*$-strongly. Note that $\|d_n\| \leq 1$ for all $n \in \N$.

We claim that $b_n - d_n \to 0$ $*$-strongly. For every $g \in G$, we have that $(1 \ot 1 \ot \lambda_g^*) a_n (1 \ot 1 \ot \lambda_g) - a_n \to 0$ $*$-strongly. Applying $\theta$, it follows that
$$(U_g^* \ot 1 \ot \lambda_g^*) \theta(a_n) (U_g \ot 1 \ot \lambda_g) - \theta(a_n) \to 0 \quad\text{$*$-strongly.}$$
Since $\theta(a_n) - b_n \ot 1 \ot 1 \to 0$ $*$-strongly, we conclude that $U_g^* b_n U_g - b_n \to 0$ $*$-strongly for every $g \in G$. The claim then follows from the dominated convergence theorem.

Since $c_n - d_n \to 0$ $*$-strongly, it follows from the claim above that $c_n - b_n \to 0$ $*$-strongly. Then also $a_n - U^*(c_n \ot 1 \ot 1)U \to 0$ $*$-strongly. Note that $U^*(c_n \ot 1 \ot 1) U = \gamma(c_n \ot 1)$. Since $(a_n)_{n \in \N}$ is a centralizing sequence in $(N \ovt M) \rtimes_\gamma G$ and since $\gamma(N \ovt 1)$ is a von Neumann subalgebra, it follows that $(c_n)_{n \in \N}$ is a centralizing sequence in $N$. Since $a_n$ asymptotically commutes with $1 \ot 1 \ot \lambda_g$, we also get that $\al_g(c_n) - c_n \to 0$ $*$-strongly for every $g \in G$. By the assumption of 2, we find a bounded sequence $t_n \in \C$ such that $c_n - t_n 1 \to 0$ $*$-strongly. Then also $a_n - t_n 1 \to 0$ $*$-strongly and 1 is proven.
\end{proof}

\begin{remark}
Consider the setting of Proposition \ref{prop.product-fullness}. We have actually proven the following two statements.
\begin{enumlist}
\item Whenever $(a_n)_{n \in \N}$ is a centralizing sequence in $N$ satisfying $\al_g(a_n) - a_n \to 0$ $*$-strongly for every $g \in G$, the sequence $a_n \ot 1$ is centralizing in $(N \ovt M) \rtimes_\gamma G$.
\item Assume that $\be$ is strictly outer and that $M \subset M \rtimes_\be G$ has stable $w$-spectral gap. Then, for every centralizing sequence $(b_n)_{n \in \N}$ in $(N \ovt M) \rtimes_\gamma G$, there exists a sequence $(a_n)_{n \in \N}$ as in 1 such that $b_n - a_n \ot 1 \to 0$ $*$-strongly.
\end{enumlist}
\end{remark}

In \cite[Theorem 6]{Jon81}, Jones proved that if $\Z \actson^\al M$ is an outer action on a full II$_1$ factor then the crossed product $M \rtimes_\al \Z$ is full if and only if the image of $\Z$ in $\Out M$ is closed. In general, assume that $G$ is a countable group and $G \actson^\al M$ is an outer action on a factor $M$ with separable predual. In \cite[Theorem A]{Mar18a}, the following two implications are proven.

\begin{enumlist}
\item If $M$ is full and the image of $G$ in $\Out M$ is closed, then $M \rtimes_\al G$ is full.
\item If $G$ is amenable and $M \rtimes_\al G$ is full, then $M$ is full and the image of $G$ in $\Out M$ is closed.
\end{enumlist}

In Theorem \ref{thm.main}, we have proven that statement~1 still holds for lcsc groups $G$. As the following example shows, statement~2 is wrong for locally compact $G$. But a weaker variant of statement~2 might still be true for lcsc groups, see question \ref{last.quest}.

\begin{example}\label{ex.full-not-full}
Let $K$ be an infinite compact abelian group and let $\Gamma \subset K$ be any countable dense subgroup. Choose a strongly continuous outer action $K \actson^\be N$ of $K$ on a full factor $N$, e.g.\ the Bogoljubov action of $K$ on $L(\F_\infty)$ associated with the left regular representation of $K$ as we recalled in Section \ref{sec.Araki-Woods}. We restrict $\be$ to an action of $\Gamma$ on $N$ and define $M = N \rtimes_\be \Gamma$. Denote $G = \Gammah$ and define $\al = \widehat{\be}$ as the dual action of $G$ on $M$. Since $N \subset M^\al$, we get that $(M^\al)' \cap M = \C 1$ and it follows that $\al$ is strictly outer. Note that $G$ is a second countable compact abelian group. The crossed product $M \rtimes_\al G \cong N \ovt B(\ell^2(\Gamma))$ is a full factor. Nevertheless, $M$ is not a full factor. This follows by taking a sequence $g_n \in \Gamma$ such that $g_n \to \infty$ in $\Gamma$ and $g_n \to e$ in $K$, so that $\al_{g_n} \to \id$ in $\Aut N$. We claim that the associated unitaries $u_{g_n} \in M$ form a nontrivial centralizing sequence in $M$. Denote by $N \subset B(H)$ the standard representation of $N$ and let $(U_g)_{g \in \Gamma}$ be the canonical implementation of the action $\be$. Then, $M = N \rtimes_\be \Gamma$ can be viewed as the von Neumann subalgebra of $B(H) \ovt L(\Gamma)$ generated by $N \ot 1$ and the unitary operators $U_g \ot \lambda_g$, $g \in \Gamma$. Since $\al_{g_n} \to \id$, we have $U_{g_n} \to 1$ $*$-strongly. Since the unitaries $1 \ot \lambda_g$ commute with $B(H) \ovt L(\Gamma)$, it follows that the sequence $U_{g_n} \ot \lambda_{g_n}$ is centralizing for $N \rtimes_\be \Gamma$.
\end{example}

\begin{question}\label{last.quest}
Let $G$ be an amenable lcsc group and let $G \actson^\al M$ be an outer action on a full factor $M$ with separable predual. Assume that $M \rtimes_\al G$ is full. Does it follow that the image of $G$ in $\Out M$ is closed?
\end{question}

\begin{remark}
Finally observe that Example \ref{ex.full-not-full} and Question \ref{last.quest} are quite subtle. In Example \ref{ex.full-not-full}, the factor $M \rtimes_\al G$ is full and thus, there is no nontrivial centralizing sequence $(a_n)_{n \in \N}$ in $M$ satisfying $\al_g(a_n) - a_n \to 0$ $*$-strongly for all $g \in G$. On the other hand, whenever $\Lambda \subset G$ is a countable subgroup, it follows from \cite[Theorem 3.3]{Mar18a} that there exists a nontrivial centralizing sequence $(a_n)_{n \in \N}$ in $M$ satisfying $\al_g(a_n) - a_n \to 0$ $*$-strongly for all $g \in \Lambda$. This can be seen as follows. Choose an ultrafilter $\om$ on $\N$ and denote by $M^\om$ the Ocneanu ultrapower. Then, $M' \cap M^\om$ is a diffuse von Neumann algebra with a faithful normal state $\vphi$ given by $\lim_{n \to \om} a_n = \vphi(a)1$ weakly for every element $a \in M' \cap M^\om$ represented by a centralizing sequence $(a_n)_{n \in \N}$. The canonical action $\Lambda \actson^\al M' \cap M^\om$ is thus $\vphi$-preserving. Since $\Lambda$ is an abelian group, it follows from \cite[Theorem 3.3]{Mar18a} that this action is not strongly ergodic. Using a diagonal argument, we then find the required centralizing sequence.

It then also follows that there exists a nontrivial centralizing net $(a_i)_{i \in I}$ in $M$ satisfying $\al_g(a_i) - a_i \to 0$ $*$-strongly for all $g \in G$. Since the dominated convergence theorem fails for nets, we cannot deduce that $(a_i)_{i \in I}$ defines a nontrivial centralizing net in $M \rtimes_\al G$, which would have been absurd.
\end{remark}

\end{document}